\documentclass{amsart}
\usepackage{amsmath,amsfonts,amssymb}
\usepackage{graphicx, epsfig,color}
\usepackage{subfig}

\newtheorem{theorem}{Theorem}[section]
\newtheorem{lemma}[theorem]{Lemma}

\newtheorem{prop}[theorem]{Proposition}
\newtheorem{defn}[theorem]{Definition}

\newtheorem{rem}[theorem]{Remark}

\newcommand{\R}{\mathbb R}
\newcommand{\Z}{\mathbb Z}

\newcommand{\LL}{\mathbb{L}}
\newcommand{\OO}{\mathbb{O}}

\mathchardef\flat="115B

\def\wt#1{\widetilde{#1}}

\def\wb#1{\overline{#1}}

\begin{document}

\title{Foliations on non-metrisable manifolds II: contrasted behaviours}

\author{Mathieu Baillif}
\address{Universit\'e de Gen\`eve, Section de Math\'ematiques, Switzerland.}
\email{Mathieu.Baillif@unige.ch}

\author{Alexandre Gabard}
\address{Universit\'e de Gen\`eve, Section de Math\'ematiques, Switzerland.}
\email{alexandregabard@hotmail.com}

\author{David Gauld}
\address{Department of Mathematics, University of Auckland, New Zealand}
\email{d.gauld@auckland.ac.nz}
\thanks{The third author was supported in part by the Marsden Fund Council from Government funding, administered by the Royal Society of New Zealand.}

\subjclass[2000]{Primary 57N99, 57R30, 37E35}

\date{\today}

\keywords{Non-metrisable manifolds, Long pipes, Foliations}

\begin{abstract}
This paper, which is an outgrowth of \cite{BGG}, continues the study of dimension 1 foliations on non-metrisable manifolds emphasising some anomalous behaviours. We exhibit surfaces with various extra properties like Type I, separability and simple connectedness, and a property which we call `squat,' which do not admit foliations even on removal of a compact (or even Lindel\"of) subset. We exhibit a separable surface carrying a foliation in which all leaves except one are metrisable but at the same time we prove that every non-metrisable leaf on a Type I manifold has a saturated neighbourhood consisting only of non-metrisable leaves. Minimal foliations are also considered. Finally we exhibit simply connected surfaces having infinitely many topologically distinct foliations.
\end{abstract}

\maketitle


\section{Introduction}\label{sec1}


In this paper we shall be concerned mainly with $1$-dimensional foliations on surfaces, so the leaves will be $1$-manifolds. Recall that there are only four of them: the circle $\mathbb{S}^1$ and the real line $\R$, which are metrisable, and the long ray $\LL_+$ and long line $\LL$, which are not. We thus call a leaf {\em short} if it is $\mathbb{S}^1$ or $\R$, and {\em long} otherwise. The closed long ray $\LL_{\ge 0}$ is the space $\omega_1\times[0,1)$, with the lexicographic order topology. We shall sometimes view $\omega_1$ as a subset of $\LL_{\ge 0}$ by identifying $\alpha\in\omega_1$ with $(\alpha,0)\in \LL_{\ge 0}$. The open long ray $\LL_+$ is $\LL_{\ge 0}\setminus\{0\}$. The long line $\LL$ is the double of $\LL_{\ge0}$: recall that the \emph{double} $2M$ of a manifold $M$ with boundary is the manifold with no boundary obtained from two copies of $M$ by identifying corresponding points on the respective boundaries; more formally the quotient of $M\times\{0,1\}$ via the equivalence relation $\sim$ generated by $(x,0)\sim(x,1)$ when $x\in\partial M$. Unlike $\R$, which is homeomorphic to $\R_+=(0,\infty)$, $\LL$ is not homeomorphic to $\LL_+$.

Our definitions of foliation, foliated chart, plaque, leaf topology and dimension and codimension of a foliation are exactly as in \cite[Definition 1.1]{BGG1}.
We shall use the word ``foliation'' to refer to foliations of dimension 1.

\medskip
Let us now summarize and motivate the results of the present paper: topics covered include infoliability, stability and minimality.

\medskip
\noindent$\bullet$ {\sc Separable surfaces lacking foliations.} 
Recall that an open (i.e. non-compact) metrisable manifold always possesses a codimension one foliation (hence also a one-dimensional foliation by Siebenmann's transversality, \cite{Siebenmann72}). In \cite[Corollary~1.3]{BGG} we listed many examples of open non-metrisable surfaces lacking foliations. (In fact \cite[Corollary~6.5]{BGG1} afforded such examples in any dimension $\ge 2$.) All the examples obtained so far were non-separable, so we pay special attention here to exhibit a separable example:

\begin{theorem}\label{thmsepnonfol}
  There is a separable squat non-metrisable surface $S$ such that for any compact subset $K\subset S$, $S\setminus K$ lacks a foliation.
\end{theorem}

A space $X$ is {\em squat} provided any continuous map $f:\LL_+\to X$ is eventually constant, that is
there is some $\alpha\in\LL_+$ such that $f$ is constant above $\alpha$. See \cite{BGG} for more details, in particular Lemma 2.2 and Theorem 4.1.

The construction starts as a variation of the famous Pr\"ufer manifold (combined with R.\,L. Moore's trick of folding boundary components), but interestingly the proof of the infoliability, though elementary, turns out to be longer than expected, and is done in Section~\ref{mixed_Pruefer_Moore}.

In Section~\ref{sep-simplconn-nofol} we provide a second example of a separable surface lacking a foliation, this time simply connected: the construction is more complicated but the proof simpler.

\medskip
\noindent$\bullet$ {\sc Surfaces lacking foliations after puncturing.} In \cite{BGG} the most basic examples of surfaces lacking foliations were bagpipes constructed from a genus $g$ surface by inserting a finite number of long pipes with a cylindrical structure ${\mathbb S}^1\times {\mathbb L}_{\ge0}$. We noticed in \cite[Proposition~6.2]{BGG1} that the removal of a single point renders all these surfaces suddenly foliable. Hence it is natural to seek an open surface which is \emph{robustly infoliable}, i.e. lacking foliations even after puncturing.  The example of Theorem \ref{thmsepnonfol} is a separable one; an $\omega$-bounded one is given by:

\begin{theorem}\label{thmpipenonfol}
  There is a long pipe $P$ such that for any compact subset $K\subset P$, $P\setminus K$ lacks a foliation.
\end{theorem}

See Section~\ref{ContortedPipe} for the construction and the definition of long pipe.

\medskip
\noindent$\bullet$ {\sc Reeb's stability.}
Inspecting the example of the punctured long plane ${\mathbb L}^2\setminus\{0\}$ foliated horizontally, one might guess that long-line leaves satisfy a form of local stability analogous to the Reeb stability theorem. 
By \emph{local stability} is meant that leaves keep the same topological type under small perturbations. Here we report the failure of any form of local stability by giving an example of Nyikos, cf.~\cite[Examples~6.3 and 6.7]{Nyikos90}. 

\begin{theorem}\label{thmonelongleaf} 
There is a separable surface carrying a foliation in which all leaves except one are real lines, the exception being the long ray. Thus neighbours of long leaves need not be long.
\end{theorem}

This manifold of Nyikos deserves special attention (regardless of our biased foliated perspective), as it unites the two simplest constructions of non-metrisable manifolds: namely Cantor's long ray (or line) versus the Pr\"ufer manifold, combining elements of the two classes of manifolds (`long' and `separable' ones). The surface is a bordification of the upper half-plane by the long ray (or the long line), so a surface-with-boundary whose interior is ${\R^2}$ and whose boundary is long (either one of the two possible long 1-manifolds). We shall construct in both cases a foliation admitting the long boundary as a leaf, but such that all other leaves are real lines. Doubling this surface provides a foliated surface corrupting the stability of long leaves (see Section~\ref{Nyikos}). So intuitively the Nyikos technique allows one to pick an individual leaf in the trivial foliation on the plane ${\mathbb R}^2$ and to make it long (in one or both directions).

However, stability in the Reeb style appears in the Type I setting for long leaves. Recall from \cite[2.10]{Nyikos84} that a space $X$ is of \emph{Type I} if it is the union of an $\omega_1$-sequence $\langle U_\alpha\ :\ \alpha<\omega_1\rangle$ of open subspaces such that each $\overline{U_\alpha}$ is Lindel\"of and $\overline{U_\alpha}\subset U_\beta$ whenever $\alpha<\beta<\omega_1$.

\begin{prop}\label{long-stability}
Suppose that a Type I manifold of dimension at least $2$ has a foliation in which there is at least one long leaf. Then there is a neighbourhood of the long leaf which intersects only long leaves.
\end{prop}

The proof can be found at the end of Section \ref{subsectionkiekie}.

\medskip
\noindent$\bullet$ {\sc Minimal foliations.} Adapting the terminology of G.\,D. Birkhoff, a foliation is \emph{minimal} if every leaf is dense.  In the non-metrisable case,  a simple example in codimension-one just comes from the Kneser foliation with a single leaf, see \cite[Section~2]{BGG1}. Here we find two examples using Nyikos's technique. Puncturing an irrationally foliated $2$-torus, one can then `stretch' the end of a leaf converging to the puncture to make it long (see Section~\ref{Mathieu's-trick}). This construction yields:

\begin{theorem}\label{miminallong}
  There is a non-metrisable surface with a minimal foliation containing exactly one long leaf (which is not embedded), the other leaves being real lines.
\end{theorem}

It is worth noting that if a leaf is sequentially compact, then it must be embedded (this follows for instance from \cite[Lemma 2.8]{GGDynamics}). This example shows that a long ray leaf is not always embedded. We also give a version with short leaves in Section \ref{minimal-short}:

\begin{theorem}\label{thmmoorised-torus} 
  There is a minimal foliation on a non-metrisable surface such that all leaves are real lines.
\end{theorem}

\medskip
\noindent$\bullet$ {\sc Counting foliations on $\omega$-bounded surfaces.}
As noticed in \cite[Corollary~6.5]{BGG}, the long plane $\LL^2$ has the striking property of having only two foliations up to homeomorphism. It is a fun `origami' game to seek a surface having exactly {\it one} foliation (up to homeomorphism). We propose the following solution using a sort of `paper cone' surface (which is $\omega$-bounded):

\begin{prop} \label{one-foliation}
The double $2\overline{Q}$ of the closed quadrant $\overline{Q}=\{(x,y)\in {\mathbb L}^2 : -y \le x\le y\}$ admits a unique foliation up to homeomorphism (up to isotopy, as well).
\end{prop}
 
We prove this in Section \ref{sec-one-foliation}. All of the examples of $\omega$-bounded simply-connected surfaces inspected in \cite{BGG} (and this last one) have a finite number of foliations up to homeomorphism (this number is often zero). Here we show that, unsurprisingly, there may be infinitely many:

\begin{prop}\label{infinitely-many-foliations}
There is a simply connected, $\omega$-bounded surface possessing infinitely many foliations (up to homeomorphism).
\end{prop}

This is proved in Section \ref{pipus-zero}. Note that in this situation all the leaves must be long lines. (Indeed the circle is excluded by Schoenflies, and the spiraling of short ends imposed by Poincar\'e-Bendixson's theory precludes real and long-ray leaves, leaving us with the long line as the only possible leaf type, see \cite{BGG, GGDynamics} for more details.)


\medskip
In the metrisable case, flows and foliations are closely related (loosely speaking, the existence of one implies the existence of the other). An analysis of which part of the `metrisable results'  survives in the non-metrisable case is done in \cite{GGDynamics}, so we will not delve into this subject except for some side remarks.

Since some of our surfaces are built by using Pr\"ufer's, Moore's and Nyikos' constructions, we recall them in Section \ref{Nyikos}, and explain how they can be done in a foliated setting.



\section{Three basic constructions: the Pr\"ufer, Moore and Nyikos surfaces}\label{Nyikos}

In this section, we recall three constructions of non-metrisable surfaces. The first two (Pr\"ufer and its Moore variant) are very classical,  but we recall them briefly since we are going to use them in foliated settings. The last one (Nyikos), though published in \cite{Nyikos90}, is apparently much less known,  so we reserve more space for its presentation, and include the proof that the resulting space is indeed a surface.

\subsection{Pr\"ufer and Moore}\label{PruferMoore}

\smallskip
The Pr\"ufer surface (with boundary) is obtained by taking  $\R\times\R_+\sqcup \cup_{x\in\R}\R_x$, where each $\R_x$ is a copy of $\R$. The topology on $\R\times\R_+$ is the usual topology. A  neighbourhood of a point $y\in\R_x$ is given by the union of an open interval $(a,b)\subset\R_x$ containing $y$ with an open sector in $\R\times\R_+$ emanating from $(x,0)$ (though this point is not in $\R\times\R_+$) whose sides have slopes $\frac{1}{a}, \frac{1}{b}$, and whose height is $\varepsilon > 0$ (see for instance \cite[Section~2]{BGG1} for more details).  With this topology, a line of slope $1/y$ that would intersect the horizontal axis at $x$ tends to the point $y\in\R_x$. The space obtained is easily seen to be a surface whose boundary is the union of all $\R_x$ (it thus has uncountably many boundary components).

Alternatively, one can take the upper half plane $\R\times\R_+$ and add at each `point' $(x,0)$ the system of rays emanating from it as a boundary component.  

There are three ways to obtain a borderless surface from the Pr\"ufer surface: add collars to each boundary component, or take the double. The third is the Moore surface which is obtained by taking the quotient of the Pr\"ufer surface by the relation $y\sim z$ for $y,z\in\R_x$ and $y=-z$, yielding a boundaryless surface. (Notice that if $y$ and $z$ belong to different $\R_x$, they will not be identified.) After the identification, each set $\R_x$ is folded around $0$ and looks more like a spike inserted `vertically' inside the upper half plane. Figure \ref{figPM} shows that if we foliate the upper half plane by vertical lines, then after `Moorizing', the leaves extend nicely through the $0$-point of each spike;
on the double of the Pr\"ufer surface, the horizontal lines together with each of the sets $\R_x$ gives a foliation. 

\begin{figure}[h]
\centering
    \epsfig{figure=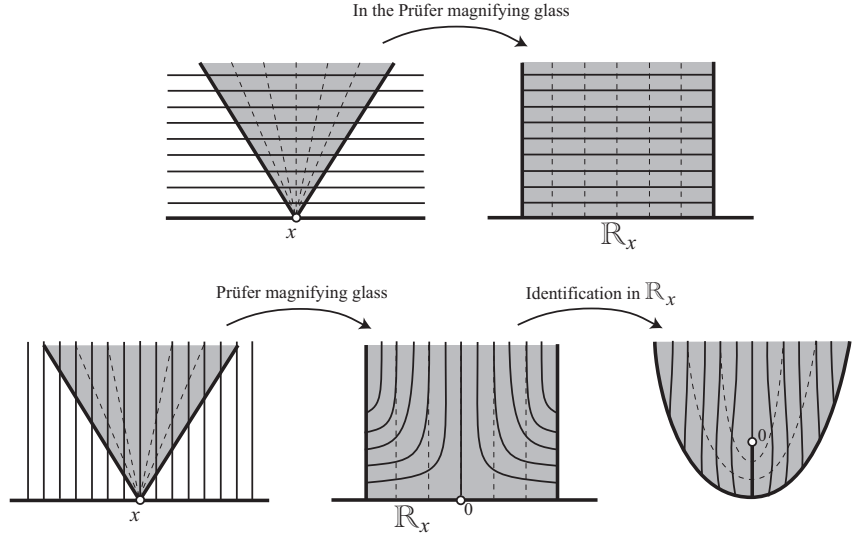, width=112mm}
    \caption{\label{figPM} Foliations on the Pr\"ufer and Moore surfaces.}
\end{figure}

We can also perform the identification in some but not all of the boundary components. Whenever we say that the boundary components in some $A\subset\R$ are {\em Moorized}, it will mean that the identification is done only in the sets $\R_x$ for $x\in A$.

Suppose we are given a surface $S$ with a chart $\varphi:\R^2\to U\subset S$, and let $S_0$ be the subsurface (with boundary) $\varphi((0,1)\times[0,1))$.\footnote{We take $(0,1)\times[0,1)$ to avoid bad behaviour near $\partial S_0$, such as a diverging sequence $x_n$ in $\R^2$ with $\varphi(x_n)$ having an accumulation point on $\partial S_0$.} We say that the {\em Pr\"uferisation of $S$ in $S_0$ through $\varphi$} is the surface obtained by first deleting the boundary of $S_0$ and replacing each of its points $x$ by copies $\R_x$ of $\R$, defining the topology as in the Pr\"ufer surface, and carrying it through $\varphi$. One can similarly define the {\em Moorization of $S$ in $S_0$}, or Moorize some of the components and not the others. For instance, if one has a foliation of the upper half plane by lines of slope $s$, we can first make them vertical by a homeomorphism, then Moorize everywhere, and carry this construction through the inverse homeomorphism. The spikes are thus inserted with slope $s$, so to say, and the foliation extends to the new space with each spike being part of a leaf. We can also Pr\"uferize or Moorize only at the points belonging to some {\em closed} subset of $\partial S_0$, and we would still obtain a surface.



\subsection{The Nyikos surface}\label{subsectionkiekie}

We shall show the following, which implies Theorem \ref{thmonelongleaf}
(compare with \cite[Examples 6.3 and 6.7]{Nyikos90}):

\begin{prop}\label{thmkiekie}
There is a surface $M=O\sqcup L$, where $O$ is homeomorphic to $\mathbb R^2\setminus\{(0,0)\}$ and $L$ to the closed long ray. 
 Moreover there is a foliation of $M$ in which all leaves except one are real, the exception being the open long ray.
\end{prop}

\begin{proof}
The version of the manifold that we present appears in \cite{Nyikos90} but the exposition is a bit sketchy, so we give a description of the manifold while describing its foliation.

It is well-known that there exists an $\omega_1$-sequence $\langle f_\alpha\rangle$ of increasing functions 
$f_\alpha:\omega\to\omega$ such that $f_\alpha<^*f_\beta$ (i.e. there is $n$ so that $f_\alpha(m)<f_\beta(m)$ for each $m\ge n$) whenever $\alpha<\beta$. 
Further, $f_0<f_\alpha$ for all $\alpha>0$ and we may assume that $f_0(n)=n$ for all $n$. 
The sequence $\langle f_\alpha\rangle$ may be constructed by induction on the limit ordinals, with $f_{\alpha+1}=f_\alpha+1$, for instance.
For each $x\in\LL_+$, write $x=\lambda+r$ for $\lambda$ a limit ordinal and $r\in[0,\infty)$ and declare 
$f_x:\omega\to[0,\infty)$ to be the function given by $f_x(s)=f_\lambda(s)+r$. If $r$ is an integer then this definition agrees with the previous one.

For each $x\in\mathbb L_{\ge0}$ define a symmetric function $g_x:\mathbb R\to(-\infty,0]$ by 
$g_x\left(\pm\frac{1}{n}\right)=-\frac{1}{f_x(n)}$, extend $g_x$ linearly on each interval 
$\pm\left[\frac{1}{n+1},\frac{1}{n}\right]$, set $g_x(0)=0$ and set $g_x(s)=g_x(1)$ when $|s|\ge 1$. 

Let $M=(\mathbb R^2\setminus\{(0,0)\})\sqcup\LL_{\ge0}$ and topologise $M$ as follows. 
For any $\varepsilon>0$ (including $\varepsilon=\infty$) and $x,y\in\LL_+$ with $x<y$ set

$$
  U(\varepsilon,x,y)=\{(s,t)\in\mathbb R^2: |s|<\varepsilon \mbox{ and }g_x(s)<t<g_y(s)\}\cup(x,y),
$$
$$
  V_\varepsilon=\{(s,t)\in\mathbb R^2 : |s|<\varepsilon \mbox{ and }-\varepsilon<t<g_\varepsilon(s)\}\cup[0,\varepsilon),
$$
where $(x,y)$ and $[0,\varepsilon)$ are both intervals in $\mathbb L_{\ge0}$.
Set
$$
  \mathcal B_1=\{ U\subset\mathbb R^2\setminus\{(0,0)\}: U \mbox{ is open in the usual topology}\},
$$
$$
  \mathcal B_2=\{U(\varepsilon,x,y) : \varepsilon>0,\ x,y\in\LL_+\mbox{ and } x<y\},
\qquad  \mathcal B_3=\{V_\varepsilon : \varepsilon>0\}.
$$
Then $\mathcal B_1\cup\mathcal B_2\cup\mathcal B_3$ is a basis for a topology on $M$. Indeed, it is clear that $U\cap V\in\mathcal B_1$ whenever $U\in\mathcal B_1$ and $V\in\mathcal B_1\cup\mathcal B_2\cup\mathcal B_3$ and that 
  $U\cap V\in\mathcal B_3$ whenever $U,V\in\mathcal B_3$. Now suppose that $\delta,\varepsilon>0$ and $v,w,x,y\in\mathbb L_+$ with $v<w$ and $x<y$. 
  Clearly $U(\delta,v,w)\cap U(\varepsilon,x,y)\cap\mathbb R_+^2\in\mathcal B_1$. Suppose $z\in(v,w)\cap(x,y)$. 
  Choose $\zeta\in\mathbb R$ so that $\zeta\le\min\{\delta,\varepsilon\}$ and
  $$
    g_{\min\{v,x\}}(s)\le g_{\max\{v,x\}}(s)<g_{\min\{w,y\}}(s)\le g_{\max\{w,y\}}(s)
  $$
  for each $s<\zeta$. Then $z\in U(\zeta,\max\{v,x\},\min\{w,y\})\subset U(\delta,v,w)\cap U(\varepsilon,x,y)$. 
  Hence $U(\delta,v,w)\cap U(\varepsilon,x,y)$ is open. Similarly $U(\delta,v,w)\cap V_\varepsilon$ is open.

\begin{rem}\label{remattach}
{\rm Although the closed long ray is included in $M$ as a subspace it is sometimes more convenient to identify the open long ray $\LL_+$ with the union of $L$ together with the negative $t$-axis of $\mathbb R^2\setminus\{(0,0)\}$.}
\end{rem}

The topological space $M$ is a surface; we note in passing that it is separable.  Moreover $\LL_+$ (in the sense of Remark \ref{remattach}) together with the lines $\{s\}\times\mathbb R$ for $s\not=0$ and $\{0\}\times(0,\infty)$ define a foliation of $M$. To verify this, we show that for each $\alpha\in\omega_1\setminus\{0\}$ there is a homeomorphism $\varphi_\alpha:{U(\infty,0,\alpha)}\to\mathbb R\times(0,1)$ which maps the interval $(0,\alpha)$ of $\LL_+$ to $\{0\}\times(0,1)$ and preserves the first coordinate of points of $\mathbb R^2$.

  Proof of the base case is very similar to proof of the successor case so we shall show only the latter of these two. 
  Suppose $\varphi_\alpha$ is given. Define $\varphi_{\alpha+1}$ as follows. 
  If $\xi\in{U(\infty,0,\alpha)}$ let $\varphi_{\alpha+1}(\xi)$ have the same first coordinate as $\varphi_\alpha(\xi)$ but half the second coordinate. 
  If $\xi\in{U(\infty,0,\alpha+1)}\setminus U(\infty,0,\alpha)$ then set $\varphi_{\alpha+1}(\xi)=(0,\frac{r+1}{2})$ if 
  $\xi=\alpha+r\in[\alpha,\alpha+1)\subset\LL_{\ge0}$; otherwise $\xi=(s,t)\in\mathbb R^2$ where $g_\alpha(s)\le t< g_\alpha(s)+1$, 
  in which case we set $\varphi_{\alpha+1}(\xi)=\left(s,\frac{1+t-g_\alpha(s)}{2}\right)$.

  Now suppose that $\lambda>0$ is a limit ordinal and that $\varphi_\alpha$ has been defined for all ordinals $\alpha\in(0,\lambda)$. 
  Choose an increasing sequence $\langle\alpha_n\rangle$ converging to $\lambda$ with $\alpha_0=0$. 
  Then $\varphi_{\alpha_n}$ has already been defined. Define sequences $\langle h_n:\mathbb R\to\mathbb R\rangle$ and $\langle N_n\rangle$ 
  inductively as follows (noting that $h_n$ and $N_n$ really depend on $\lambda$ and $\alpha_n$ as well). 
  The functions $h_n$ are symmetric so we need only define $h_n(s)$ for $s>0$. 
  Set $N_0=1$ and $h_0=g_0$. Given $N_{n-1}$ and $h_{n-1}$, choose $N_n$ so that $N_n-N_{n-1}>1$ and 
  $g_{\alpha_{n-1}}(s)<g_{\alpha_n}(s)<g_\lambda(s)$ whenever $0<s<\frac{1}{N_n}$, and let $h_n$ agree with 
  $g_{\alpha_n}$ on $\left(0,\frac{1}{N_n+1}\right]$, let $h_n(s)=\frac{h_{n-1}(s)+g_\lambda(s)}{2}$ when $s\in\left[\frac{1}{N_n},\infty\right)$ 
  and extend $h_n$ linearly in $\left[\frac{1}{N_n+1},\frac{1}{N_n}\right]$. 
  Inductively $h_n(s)=g_{\alpha_n}(s)$ when $|s|<\frac{1}{N_n}$, and for any $s\not=0$ the sequence $\langle h_n(s)\rangle$ is strictly 
  increasing and converges to $g_\lambda(s)$.

  Let $B_n=\varphi_{\alpha_n}\left(\{(s,t)\in\mathbb R\setminus\{(0,0)\}\times(-\infty,0): h_{n-1}(s)\le t<h_n(s)\}\cup[\alpha_{n-1},\alpha_n)\right)$ 
  (but with $[\alpha_0,\alpha_1)$ replaced by $(\alpha_0,\alpha_1)$ when $n=1$) and let 
  $\psi_n:B_n\to\mathbb R\times\left[1-\frac{1}{2^{n-1}},1-\frac{1}{2^n}\right)$ 
  ($\left[0,\frac{1}{2}\right)$ replaced by $\left(0,\frac{1}{2}\right)$ when $n=1$) be a homeomorphism which preserves the first coordinate and, 
  for each fixed value of the first coordinate, the second coordinate is increasing. 
  Then we may define $\varphi_\lambda|\varphi_{\alpha_n}^{-1}(B_n)$ to be $\psi_n\varphi_{\alpha_n}$.

Because the $s$-coordinate does not change at all, the charts $\varphi_\alpha$ are foliated charts, confirming our claim about the foliation.
\end{proof}
\medskip

\begin{rem}\label{remafterNyikos} 
{\rm  The construction above can be varied in a number of ways.  We could have performed the construction only on the half plane $s\ge0$ to obtain a surface with boundary $\mathbb L_+$, or only on the half plane $t<0$ so that the long leaf projects from the edge of the manifold. We can combine two surfaces as above to obtain a surface foliated by leaves homeomorphic to $\mathbb R$ except for one leaf which is homeomorphic to $\mathbb L$; alternatively we can append a copy of $\mathbb R\times\mathbb L_{\ge0}$ to obtain a surface foliated by leaves homeomorphic to $\mathbb L_+$ except for one leaf which is homeomorphic to $\mathbb L$. 
}
\end{rem}

As in the case of Pr\"ufer and Moore, the Nyikos construction (and its modifications just mentioned) can be carried out on a surface $S$ different from the plane, provided that we are given a homeomorphism $\varphi: \R^2\to U\subset S$ sending $0$ to $p$. We will then say that the resulting surface is the {\em Nyikosization of $S$ in $U$ through $\varphi$}.

\smallskip
We shall now prove that this construction cannot be adapted to obtain a Type I manifold, that is, we prove
Proposition \ref{long-stability}.

\begin{proof}[Proof of Proposition \ref{long-stability}]
  Let $M=\cup_{\alpha<\omega_1}U_\alpha$, where each $U_\alpha$ is open and Lindel\"of and $\overline{U_\alpha}\subset U_\beta$ when $\alpha<\beta$. 
  Suppose for simplicity that $L$ is a long leaf. Choose an embedding $e:\LL_+\to L$: for convenience we assume that 
  $e(\alpha)\notin\overline{U_\alpha}$. For each $\alpha<\omega_1$ by \cite[Lemma 2.3]{BGG} there is a foliated chart $(V_\alpha,\varphi_\alpha)$ with 
  $e([0,\alpha])\subset V_\alpha$. Because $M$ is first countable, there is a neighbourhood $N$ of $e(0)$ which lies in 
  uncountably many of the tubes $V_\alpha$ and such that all leaves passing through $N$ meet the neighbourhood 
  $M-\overline{U_\alpha}$ of $e(\alpha)$. Thus $N$ cannot intersect a short leaf. Indeed, if $F$ is a short leaf intersecting $N$ then $F$ is 
  Lindel\"of, so there is $\alpha\in\omega_1$ so that $F\subset U_\alpha$. 
  Thus all leaves through $N$ are long. Of course continuum many leaves pass through $N$.
\end{proof}

Note that we cannot ask that the long leaves should all be of the same type. For example, if we foliate the long plane $\mathbb L^2$ by horizontal long lines then remove a single point, the resulting foliation has two leaves which are long rays while the remainder are long lines. On the other hand if the leaf $L$ in the proof above is long in both directions, i.e., modelled on $\mathbb L$, then the argument may be applied in both directions of $L$ to exhibit a neighbourhood which intersects only leaves which are long in both directions.


\section{Dense leaves, minimal foliations on non-metrisable surfaces}\label{dense}

\subsection{A minimally foliated surface with a long ray leaf that is not embedded}\label{Mathieu's-trick}

In this sub-section we prove Theorem \ref{miminallong}. Our surface has fundamental group free on two generators. We then show that it is impossible to construct such a surface which is also simply connected.

Let ${\mathbb T}^2={\mathbb R}^2/{\mathbb Z}^2$ be the $2$-torus irrationally foliated by a `Kronecker' foliation, that is, its leaves are (quotients of) the lines of irrational slope $a$ in $\R^2$. Let $p\in\mathbb{T}^2$, and choose a small neighbourhood of $p$ given by a parallelogram with two horizontal sides and two sides of slope $a$. Then there is a homeomorphism $\varphi: \R^2\to U$ sending $0$ to $p$, the horizontals to horizontal lines and the vertical lines to the leaves.

Take $S$ to be the Nyikosization of ${\mathbb T}^2$ in $U$ through $\varphi$. This process can be seen as first making a puncture at $p$, whose leaf is disconnected in two parts, the negative and positive ones. Since $\varphi$ sends the vertical lines to leaves, the added copy of $\LL_+$ gets attached to the negative part of the leaf (recall Remark \ref{remattach}), and the foliation extends to $S$. Then both the `augmented' negative part and the positive one remain dense in the resulting surface, yielding our result.

\begin{prop} A simply connected surface does not support a foliation having a dense leaf.
\end{prop}
\begin{proof}
Suppose that $S$ is a simply connected surface and $\mathcal F$ is a foliation on $S$ with a leaf $L$ which is dense in $S$. Give $L$ an orientation. Choose a foliated chart $(U,\varphi)$ with $\varphi(U)=(0,3)\times(0,3)$ and set $V=\varphi^{-1}((0,3)\times(1,2))$. Next choose $y_1,y_2\in(1,2)$ such that:
\begin{itemize}
\item the plaques $\varphi^{-1}((0,3)\times\{y_1\})$ and  $\varphi^{-1}((0,3)\times\{y_2\})$ lie in $L$;
\item the natural orientations (i.e., as inherited from the direction of the $x$-axis of $\R^2$) of $\varphi^{-1}((0,3)\times\{y_1\})$ and  $\varphi^{-1}((0,3)\times\{y_2\})$ either both agree with the orientation of $L$ or are both opposite to it;
\item there is no other plaque in $U$ and lying on $L$ whose location on $L$ is between the plaques $\varphi^{-1}((0,3)\times\{y_1\})$ and $\varphi^{-1}((0,3)\times\{y_2\})$.
\end{itemize}
We describe a procedure for obtaining $y_1$ and $y_2$. The leaf $L$ meets $V$; choose $y\in(1,2)$ so that $\varphi^{-1}((0,3)\times\{y\})\subset L$. Let us assume that the natural orientation of $\varphi^{-1}((0,3)\times\{y\})$ agrees with that of $L$ and follow $L$ in its positive direction until it next meets $V$, say in $\varphi^{-1}((0,3)\times\{z\})$. If the natural orientation of $\varphi^{-1}((0,3)\times\{z\})$ also agrees with that of $L$ then set $\{y_1,y_2\}=\{y,z\}$. Otherwise continue to follow $L$ in its positive direction until it first meets $\varphi^{-1}((0,3)\times(y,z))$ (if $y<z$; otherwise use $(z,y)$), say in $\varphi^{-1}((0,3)\times\{w\})$. If the natural orientation of $\varphi^{-1}((0,3)\times\{w\})$ agrees with that of $L$ then let $\{y_1,y_2\}=\{w,y\}$ and if not let $\{y_1,y_2\}=\{w,z\}$.

\begin{figure}[h]
\centering
    \epsfig{figure=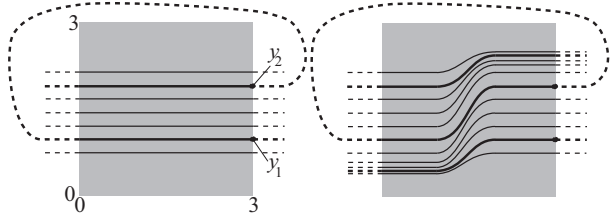,width=80mm }
    \caption{\label{Following a leaf} Following a leaf to create a circular leaf}
\end{figure}

Let $h:(0,3)\times(0,3)\to(0,3)\times(0,3)$ be a homeomorphism which is the identity except on $(1,3)\times(1,2)$ and maps sets of the form $(2,3)\times\{y\}$ to sets of the form $(2,3)\times\{\eta\}$, with $\eta=y_2$ when $y=y_1$.  See figure \ref{Following a leaf}.

Now let $\mathcal B$ be a basis for $\mathcal F$ in which the only chart meeting $\varphi^{-1}([1,2]\times[1,2])$ is the chart $(U,\varphi)$. Define a new foliation $\mathcal F'$ on $S$ which has basis $\mathcal B\setminus\{(U,\varphi)\}\cup\{(U,h^{-1}\varphi)\}$. The foliation $\mathcal F'$ has a leaf $L'$ which consists of the portion of $L$ between the two plaques plus the new plaque $\varphi^{-1}h((0,3)\times\{y_1\})$. Now the leaf $L'$ is a circle so by \cite[Theorem 1]{GaGa} $L'$ bounds a disc because $S$ is simply connected. Restricting $\mathcal F'$ to this disc gives a foliation of the disc in which the boundary is a leaf, which is impossible.
\end{proof}

Building on this argument, one can adapt the full Haefliger-Reeb theory to the non-metric realm. We refer the interested reader to \cite{Gabard11}.

\subsection{A minimally foliated non-metric surface with only short leaves}\label{minimal-short}

In this sub-section we prove Theorem \ref{thmmoorised-torus}.

\begin{proof}[Proof of Theorem \ref{thmmoorised-torus}]
We start again with a torus and an irrational foliation. Here, it is convenient to treat the torus as $[0,1]\times [-1/2,1/2]$, with the obvious identifications. Define $K$ to be the $1/10$ Cantor set, that is, the set of points in $[0,1]$ which may be written with decimals using only $0$ and $9$. Of course there are sometimes two ways of writing a decimal number, but when we have the choice we always choose to use the notation with infinitely many $9$s. (So for instance $0.1$ is in $K$, as it is equal to $0.0\bar{9}$.) Now, notice that if a real number $a$ may be written in decimals with one of its decimals (say the $n^{\rm th}$) in $\{3,4,5,6,7\}$, then we have:
\begin{equation}\label{axnotinK}
  a + x \textrm{ is not in $K$ for all $x$ in $K$ },
\end{equation}
because the $n^{\rm th}$ decimal of $a + x$ would then be in
$\{2,3,4,5,6,7,8\}$.

\begin{lemma}\label{aexists}
  There is an irrational number $a$ such that for all $k\in\Z\setminus\{0\}$, there is some $n(k)$ such that the $n(k)^\text{th}$ decimal of $ka$ is in $\{3,4,5,6,7\}$.
\end{lemma}
\proof
  For each $k\in \Z\setminus\{0\}$,
  let $U_k=\{x\in(0,\infty)\,:\, $ there is a $5$ somewhere after the decimal point in $kx$
  but the decimal expansion of $x$ does not end in either $5000\dots$ or $5999\dots\}$.
  The set $U_k$ is open and dense in $(0,\infty)$, so by the Baire Category Theorem, 
  the intersection of all the sets $U_k$ is a dense $G_\delta$. Hence $\cap_{k\in\Z\setminus\{0\}}U_k$ must contain at least one irrational number, call it $a$. 
  Then for each $k\in\Z\setminus\{0\}$ there is some $n(k)$ such that the $n(k)^\text{th}$ decimal of $ka$ is $5$.
\endproof

Let $a$ be a real number given by Lemma \ref{aexists}, and foliate the torus by the irrational lines of slope $1/a$. Then
for each leaf $F$ in the foliation, the intersection of $F$
and $K\times\{0\}$ is at most one point.

Choose $\varepsilon\in(0,\frac{1}{4a})$: then the leaf passing through the point $(1/2, 0)$ meets the horizontal line $(0,1)\times\{2\varepsilon\}$ before the vertical line $\{1\}\times(-1/2,1/2)$. Let $U$ be the open subset of the torus obtained from the strip $[0,1]\times(-2\varepsilon,2\varepsilon)$ from which the (part of the) leaf passing through $(1/2, 0)$ has been removed. Let $V\subset U$ be obtained from the strip $[0,1]\times(-\varepsilon,\varepsilon)$ with all the (parts of) leaves passing through the points in $[1/3,2/3]\times\{0\}$ removed. Notice that $V$ contains $K\times\{0\}$. Set $V^-=V\cap[0,1]\times[-1/2,0]$ and $V^+=V\cap[0,1]\times[0,1/2]$. Then there is a foliated chart $\varphi$ sending $U$ to $\R^2$ and $V$ to $(-1,1)^2$. Assume $V^-$ is sent to $(-1,0]\times(-1,1)$ and $V^+$ to $[0,1)\times(-1,1)$. Replace each point $x$ of $K\times\{0\}$ by two intervals $I_x$, $J_x$ by Pr\"uferizing in $V^-$ and $V^+$ through $\varphi$, and then Moorize $I_x$ and $J_x$. This yields a foliated surface (we use here that $K\times\{0\}$ is closed, of course). Since $K$ is uncountable, this surface is non-metrisable. Leaves that used to pass through $K\times\{0\}$ are now cut into $2$ leaves, and these remain dense in the resulting surface by (\ref{axnotinK}). We thus have a minimal surface. 
\end{proof}

One might think at first that we can Moorize all leaves by carefully choosing $K$, but this is impossible. Indeed,  for $n\in \mathbb{Z}$, call $K_n$ the $n$-th `translate' of $K$, consisting of the points $(x + na ($mod$1), 0)$, for $x\in K$ (these are the points obtained by following the foliation). Then, the $K_n$ are mutually disjoint. If $K$ has a non-empty interior, then it contains an interval of length $\varepsilon$, say, and so does each $K_n$ (we think of them as intervals in the circle). Then, the $K_n$ for $n = 0$ to $1/\varepsilon + 1$ cannot be mutually disjoint, because this would imply that some intersect $K$ more than once. Hence $K$ has an empty interior. If $K$ is closed (and so is each $K_n$), this implies that each $K_n$ is nowhere dense. It follows from Baire's Theorem that $[0,1]$ cannot be the union of all $K_n$ for $n \in \mathbb{Z}$.

We are obliged to modify only some of the leaves, which fall thus in two categories after the Moorization: the ones that are dense only at `one end', and the ones for which `both ends' are dense.



\section{Surfaces without foliations even after removing a compact subset}

\subsection{A contorted long pipe}\label{ContortedPipe}

Here we show that there is a long pipe that cannot be foliated even after removing any compact subset, that is, we prove Theorem~\ref{thmpipenonfol}. Our definition of \emph{long pipe} is the same as \cite[5.2]{Nyikos84}: the union of an $\omega_1$-sequence $\langle U_\alpha\ :\ \alpha<\omega_1\rangle$ of open subspaces $U_\alpha$ each homeomorphic to $\mathbb S^1\times\R$ such that $\overline{U_\alpha}\subset U_\beta$ and the boundary of $U_\alpha$ in $U_\beta$ is homeomorphic to $\mathbb S^1$ whenever $\alpha<\beta<\omega_1$.

\begin{proof}[Proof of Theorem~\ref{thmpipenonfol}]
In fact, we construct a long pipe $N=\cup_{\alpha\in\omega_1}N_\alpha$ such that for any $\alpha\in\omega_1$, $N\setminus\wb{N_\alpha}$ cannot be foliated.  Since any compact (indeed Lindel\"of) subset of $N$ is contained in some $N_\alpha$, the result follows. The idea is the following. We start with the first octant $\OO=\cup_{\alpha\in\omega_1}\OO_\alpha$, where
$\OO_\alpha=\{(x,y)\in\LL^2\,:\, 0\le y\le x<\alpha\}$, and we glue the boundaries together, that is we identify $(x,0)$ with $(x,x)$ for each $x\in\mathbb L_+$. Call this space $C=\cup_{\alpha\in\omega_1}C_\alpha$. By \cite[Proposition~6.1]{BGG}, for any $\alpha$, $C\setminus C_\alpha$ cannot be foliated by long leaves. After a puncture (say at $(0,0)$) $C$ can be foliated by short leaves, for instance the circles $\{(x,y)\,:\, x= \mbox{constant}\}$. To prevent this, we alter the boundaries $\overline{C_\beta}\setminus C_\beta$ (that are homeomorphic to the circle) at limit ordinals $\beta$ by inserting intervals so that they become homeomorphic to the circle with a spike, i.e. to the subset of $\mathbb R^2$ given by $\{(x,y)\,:\,x^2+y^2=1\mbox{ or } x\ge 0, y=0\}$. By \cite[Proposition~4.2]{BGG}, if one foliates this new space by short leaves, there would be a limit ordinal $\beta$ such that the altered boundary at height $\beta$ is saturated by the foliation, but this is impossible as this boundary is not even a manifold.

Denote by $\Lambda\subset\omega_1$ the limit ordinals. For each $\gamma\in\Lambda$, let $I_\gamma$ be a copy of $[3,4)$, and let $\langle\beta_{\gamma,n}\rangle$ be an increasing sequence of ordinals converging to $\gamma$. For each $\alpha\in\omega_1$, let $O_\alpha=\mathbb O_\alpha\setminus\{(0,0)\}$ and set $\widetilde{O}_\alpha=O_\alpha\cup(\bigcup_{\beta<\alpha; \beta\in\Lambda}I_\beta$. For each $\alpha\in\omega_1$, choose a homeomorphism $\varphi_\alpha:[2,\alpha)\to[2,4)$ such that if $\alpha\in\Lambda$ then $\varphi(\beta_{\alpha,n})=4-1/n$ when $n\ge 1$. Let $B$ be $(0,4)\times[0,2]$, and fix an embedding  $g:B\to B$ which is the identity outside $[2,4)\times(0,1]$ and whose image is $B \setminus [3,4)\times\{1/2\}$, such that $g$ leaves the second coordinate fixed, as in Figure \ref{noFoliation}. Next, for each $\alpha\in\omega_1\setminus\{0,1,2\}$ choose homeomorphisms $\psi_\alpha:O_\alpha\to B$ such that $\psi_\alpha|_{\overline{O_2}}$ does not change the $x$-coordinate, $\psi_\alpha|_{O_\alpha\setminus O_2}$ changes the first coordinate $x$ to $\varphi_\alpha(x)$, and $\psi_\alpha |_{[2,\alpha)\times[0,1]}$ is given by $(x,y)\mapsto (\varphi_\alpha(x),y)$.

\begin{figure}[h]
\centering
    \epsfig{figure=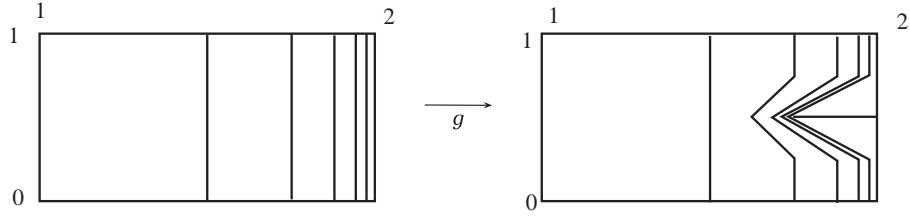, width=120mm}
    \caption{\label{noFoliation} Inserting intervals}
\end{figure}

We now define a topology on $\cup_{\alpha\in\omega_1}\widetilde{O}_\alpha$ by use of functions $\widetilde{\psi}_\alpha:\widetilde{O}_\alpha\to B$ ($\alpha>2$) which are defined by induction. If $\alpha < \omega$, we let
$\widetilde{\psi}_\alpha =\psi_\alpha$. If $\alpha\in\Lambda$, we let
$\widetilde{\psi}_\alpha|_{\widetilde{O}_{\beta_{\alpha,n}}\setminus\widetilde{O}_{\beta_{\alpha,n-1}}}$ 
be given by
$\psi_\alpha\circ(\psi_{\beta_{\alpha,n}})^{-1}\circ\widetilde{\psi}_{\beta_{\alpha,n}}$.
If $\alpha=\gamma+1$ and $\gamma\notin\Lambda$, we let
$\widetilde{\psi}_\alpha|_{\widetilde{O}_\gamma}$ be
$\psi_\alpha\circ(\psi_\gamma)^{-1}\circ\widetilde{\psi}_\gamma$,
and 
$\widetilde{\psi}_\alpha|_{\widetilde{O}_\alpha \setminus \widetilde{O}_\gamma}$ be $\psi_\alpha|_{\widetilde{O}_\alpha\setminus \widetilde{O}_\gamma}$. 
If $\gamma\in\Lambda$, we let
$\widetilde{\psi}_\alpha|_{\widetilde{O}_\alpha \setminus \widetilde{O}_\gamma}$ be $\psi_\alpha|_{\widetilde{O}_\alpha \setminus \widetilde{O}_\gamma}$,
$\widetilde{\psi}_\alpha(x)=\psi_\alpha\psi_\gamma^{-1}(x,1/2)$ for $x\in I_\gamma$, and
$\widetilde{\psi}_\alpha|_{\widetilde{O}_\gamma}$ be
$\psi_\alpha\circ(\psi_\gamma)^{-1}\circ g\circ\widetilde{\psi}_\gamma$. 
Choose the topology on $\widetilde{O}_\alpha$ that makes $\widetilde{\psi}_\alpha$ a
homeomorphism. Then
$\widetilde{O}=\cup_{\alpha\in\omega_1}\widetilde{O}_\alpha$
is a surface with boundary, which is a version of $\OO$ with inserted spikes. Notice that the boundary is not affected by the spiking process, so we can write these points using the same coordinates as in $\OO$. The required surface $N$ is obtained by identifying $(x,0)$ with $(x,x)$. The proof of \cite[Proposition~6.1]{BGG} adapts to the case of $N$, and therefore there cannot be any foliation of $N\setminus\wb{N_\alpha}$ by long leaves. Since we inserted the spikes, by \cite[Proposition~4.2]{BGG} neither can there be a foliation by short leaves, so $N\setminus\wb{N_\alpha}$ cannot be foliated, for any $\alpha$.
\end{proof}


\subsection{The mixed Pr\"ufer-Moore surface}\label{mixed_Pruefer_Moore}

Here, we prove Theorem~\ref{thmsepnonfol}, that is the existence of a separable non-metrisable squat surface lacking a foliation, even after removal of a compact subset.

Our example is a mixed Pr\"ufer-Moore surface. Let $P$ be the (separable) Pr\"ufer surface with boundary described in Section~\ref{PruferMoore}, i.e. the upper half plane whose horizontal axis has been Pr\"uferized. To avoid confusion, we shall denote the boundary components by $I_r$ ($r\in\R$) instead of $\R_r$ or $\R_x$, indexed according to their horizontal coordinate. Denote by $P^o$ the interior of $P$ (i.e. the upper half plane). For $A,B\subset\R$ disjoint, let $P_{A,B}=P^o\cup\bigcup_{r\in A\cup B}I_r$ be the surface with boundary where we Moorize the $I_r$ for $r\in B$. Notice that $P_{A,B}$ and its double $2P_{A,B}$ are squat (see \cite[Example 4.4]{BGG1}), separable, and moreover non-metrisable whenever one of $A,B$ is uncountable. We shall show that for certain $A,B$, the double $2P_{A,B}$ does not have any non-trivial foliation, even upon removal of a compact set.

\begin{theorem}\label{thm31bis}
Let $A,B\subset\R$ be disjoint, and $U\subset \R$ be open. Assume that $B$ is a $G_\delta$ dense subset of $U$ and that $A$ is such that $A\cap U_0$ is uncountable for each non-empty open $U_0\subset U$. Let $K\subset 2P_{A,B}$ be compact. Then $2P_{A,B}\setminus K$ does not possess any foliation.
\end{theorem}

For example, choose $B$ to be a dense $G_\delta$ of Lebesgue measure zero (such as the Liouville numbers), and set $A$ to be its complement. The theorem will follow from a series of lemmas, some of which do not use all the assumptions, so for now we only assume that $A,B$ are subsets of $\R$. We begin with some notation.

When $r\in A$, we still write $I_r$ for the subsets of $2P_{A,B}$ corresponding to the former boundaries that are now `bridging' the two copies of $P_{A,B}$. When $r\in B$ we write $I_{r,1},I_{r,2}$ for the two copies of $I_r$ in $2P_{A,B}$. We write $P^o_1,P^o_2$ for the copies of $P^o$. When $\mathcal{F}$ is a foliation on $2P_{A,B}$ we let
\begin{align*}
  D &=\{r\in A\,:\, I_r\textrm{ is not a leaf of }\mathcal{F}\},\\
  E_i &=\{r\in B\,:\, I_{r,i}\textrm{ is not contained in a single leaf of }\mathcal{F}\},\, i=1,2.
\end{align*}

\begin{lemma}\label{31bis1}
  Let $A,B\subset\R$, and $\mathcal{F}$ be a foliation on $2P_{A,B}$. Then $D,E_1,E_2$ are all countable.
\end{lemma}

\begin{proof}[Proof of Lemma \ref{31bis1}] Since $2P_{A,B}$ is squat, all the leaves are short, and thus a leaf can intersect only countably many sets $I_r$. Let $\{F_n\}_{n\in\omega}$ be a set of leaves dense in $P^o_1\cup P^o_2$, and set
$$
  D'=\{r\in A\,:\,\exists n\in\omega
  \textrm{ with }F_n\cap I_r\not=\emptyset\}.
$$
We show that $D'=D$, which proves the lemma for $D$. First, notice that
$$
  D=\{r\in A\,:\, \textrm{for some leaf }F,\ F\cap I_r\not=\emptyset\not=F\cap (P^o_1\cup P^o_2)\}.
$$
Let $F$ be a leaf that intersects both $P^o_1\cup P^o_2$ and some $I_r$. Fix $x\in F\cap (P^o_1\cup P^o_2)$, and follow the leaf (in one direction or the other) until it reaches $I_r$. Let $y$ be the first point in $I_r$ thus reached, and let $z\in F$ be a point further on. Any neighbourhood of $y$ contains a point of $I_r$ that is {\em not} in $F$ (see Figure \ref{figureI_r}, left). By \cite[Lemma 2.3]{BGG}, there is a `tubular neighbourhood' $T$ of the portion of $F$ between $x$ and $z$ with a homeomorphism $T\to(0,1)^2$ such that the leaves in $T$ are the preimages of the horizontals. Therefore, there is some $n$ for which $F_n$ intersects $I_r$ as well (see Figure \ref{figureI_r}, right). It follows that $D'=D$.

\begin{figure}[h]
\centering
    \epsfig{figure=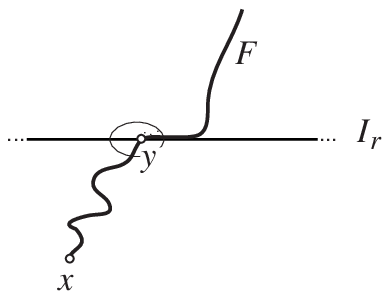, width=40mm}\hskip 1cm
    \epsfig{figure=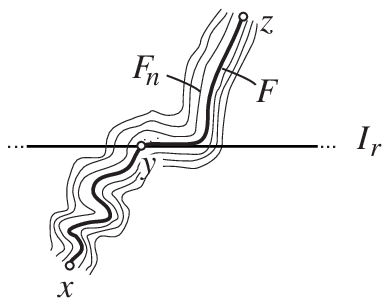, width=40mm}
    \caption{\label{figureI_r} Leaf intersecting $I_r$ and its tubular neighborhood.}
\end{figure}

Now, fix $i$. For $r\in B$, there is a homeomorphism $h:I_{r,i}\to[0,1)$. Denote by $G_r$ the leaf passing through $h^{-1}(0)$. Notice that
$$
  E_i=\{r\in B\,:\, \textrm{for some leaf }F,\ F\cap I_{r,i}\not=\emptyset \textrm{ and }F\not= G_r\}.
$$
Define $\mathcal{G}$ to be the set of leaves different from all $G_r$ but intersecting some $I^i_r$. Taking a countable subset of leaves dense in $(P^o_1\cup P^o_2)\cap(\cup_{F\in\mathcal{G}}F)$, one sees as in the proof for $D$ that $E^i$ must be countable.
\end{proof}

This lemma enables us to look only at {\em nice} foliations, i.e., those satisfying:
\begin{align}
  I_r&\textrm{ is a leaf of $\mathcal{F}$ if }r\in A \label{IrA}\\
  I_{r,i}&\textrm{ is contained in a single leaf of $\mathcal{F}$ if }r\in B,\,i=1,2. \label{IrB}
\end{align}

We now look at each copy of $P_{A,B}$ separately, and drop the superscripts on $I^i_r$. We define $W_r(n,\theta)$ for $r\in\R$, $n\in\omega$ and $\theta\in (0,\pi)$ to be the interior of the triangle shown in Figure \ref{figure4}. $W_r(n,\theta)$ is centred around the line of points with first coordinate $r$, has height $1/n$ and lower vertex $(r,0)$: by a slight abuse of language, we shall say that $r$ is the bottom vertex of $W_r(n,\theta)$, even if the point $(r,0)$ has been replaced by $I_r$ in the Pr\"uferisation.

\begin{figure}[h]
\centering
    \epsfig{figure=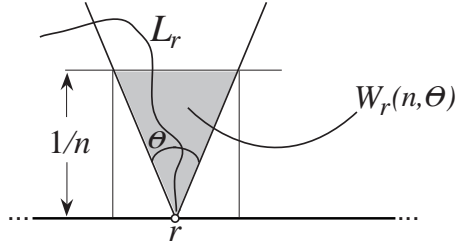, width=60mm}
    \caption{\label{figure4} $W_r(n,\theta)$.}
\end{figure}

Let $T(r,n,\theta)$ denote the following property (compare Figure \ref{figure4}): 
\begin{quote} There is a leaf $L_r\in\mathcal{F}$ such that $L_r\cap W_r(n,\theta)$ contains an arc that goes from the bottom vertex of $W_r(n,\theta)$ to its upper boundary.
\end{quote}

\begin{lemma}\label{31bis3}
Let $A,B\subset\R$ and $\mathcal{F}$ be the restriction to one of the copies of $P_{A,B}$ of a nice foliation on $2P_{A,B}$. For each $r\in B$ and $\theta\in(0,\pi)$ there is some $n\in\omega$ such that $T(r,n,\theta)$.
\end{lemma}

\begin{proof}
For $r\in B$, let $F_r$ denote the leaf containing $I_r$. Then $F_r$ must converge to the $0$-point of $I_r$. The result then follows from the definition of the topology near this point.
\end{proof}

Let $B(n,\theta)=\{r\in A\cup B\,:\,T(r,n',\theta)\text{ for some }n'\le n\}$.

\begin{lemma}\label{31bis4}
With the same assumptions as in Lemma \ref{31bis3}, we have that $\wb{B(n,\theta)}\cap (A\cup B)\subset \cup_{m\in\omega}B(m,\theta')$ for all $\theta' > \theta$.
\end{lemma}

\begin{proof}
Notice that Lemma \ref{31bis3} implies Lemma \ref{31bis4} for points in $B\cap\wb{B(n,\theta)}$, so we only need to check it for points in $A\cap\wb{B(n,\theta)}$. We shall consider induced foliations on various subspaces of $P_{A,B}$, so we will call a leaf of the induced foliation on $X\subset P_{A,B}$ an $X$-leaf. Let $n\in\omega$ and $\theta\in(0,\pi/2)$, and $r\in\wb{B(n,\theta)}\cap (A\cup B) - B(n,\theta)$. (Since $B(n,\theta)\subset B(n,\theta')$ for all $\theta'>\theta$, this is the only case we need to check.) Fix $\theta'>\theta$. We can assume without loss of generality that there is a sequence $\langle r_i\rangle$ in $B(n,\theta)$, $i\in\omega$, converging to some $s\in A$, with $r_i<r_{i+1}$ $\forall i$. Denote the corresponding triangles by $W_{r_i}(n,\theta)$. Figure \ref{figure5} illustrates the situation.
 
\begin{figure}[h]
\centering
    \epsfig{figure=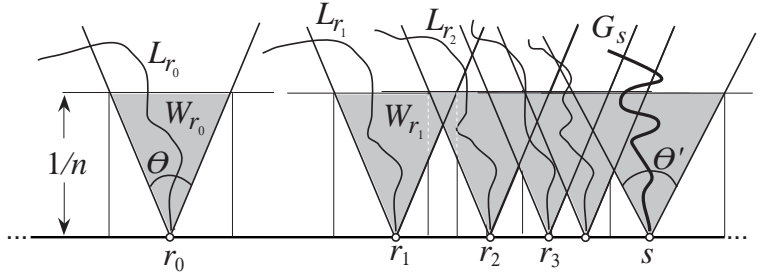, width=100mm}
    \caption{\label{figure5} An accumulation of sectors.}
\end{figure}

Our goal is to find a leaf $G_s$ going through all of $W_s(m,\theta')$, as in Figure \ref{figure5}, for some $m\ge n$ (we drew the case where $m=n$). Denote by $S$ the strip $\R\times (0,1/n)$ in $P^o$ (viewed again as the upper half plane). For each $i$, choose an $S$-leaf $F_i\subset L_{r_i}\cap W_{r_i}(n,\theta)$ going through all of the triangle. ($L_i\cap S$ may have more than one component going through all the triangle, we choose one as $F_i$.) Consider $Y=\wb{\cup_{i\in\omega} F_i} \setminus \cup_{i\in\omega} F_i$, the set of accumulation points of sequences $\langle z_i\rangle_{i\in\omega}$ with $z_i\in F_i$. Then $Y\cap S$ is a union of $S$-leaves, and $Y\subset \wb{W_r(n,\theta)}$, so $Y\cap S\subset W_r(n,\theta')$. Moreover, $Y$ contains points at each height in $(0,1/n)$. Since $Y\subset \wb{W_r(n,\theta)}$, applying some Poincar\'e--Bendixson theory shows that the $S$-leaves in $Y$ must escape from $W_r(n,\theta)$ by either its upper boundary or its bottom vertex (which is really a line segment in $P_{A,B}$).
   
We return to the foliation on all of $2P_{A,B}$, and consider $W_s(n,\theta')\cup\wt{W_r}(n,\theta')$, where $\wt{W_s}(n,\theta')$ denotes the copy of $W_s(n,\theta')$ in the other copy of $P_{A,B}$. The closure in $2P_{A,B}$ of this union is homeomorphic to a rectangle of width $\theta'$ and height $2/n$ (see Figure \ref{figures8-9}, left).

\begin{figure}[h]
\centering
    \epsfig{figure=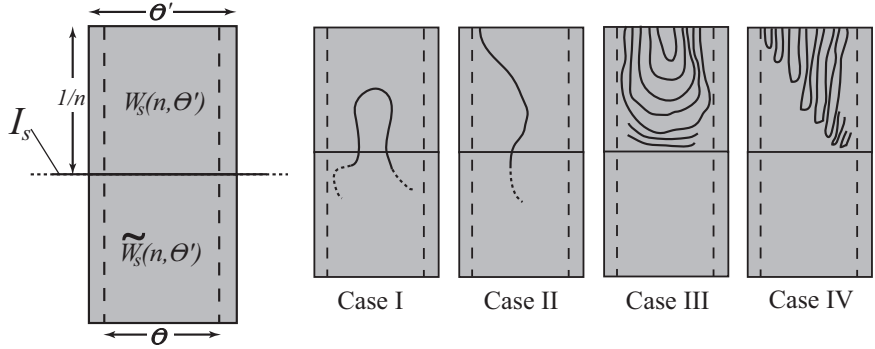, width=115mm}
    \caption{\label{figures8-9} Leaves inside $W_s(n,\theta')$.}
\end{figure}

Since $Y$ contains points at each height, we have the $4$ cases depicted in Figure \ref{figures8-9} (right): 
\begin{itemize}
     \item[--] There is a leaf which exits $W_s(n,\theta')$ twice by its bottom side (or as we wrote above, by the vertex). (Case I.) Then for some $m\ge n$, $r\in B(m,\theta')$.
     \item[--] There is a leaf going right through $W_s(n,\theta')$ (Case II), which yields $r\in B(n,\theta')$.
     \item[--] There is a sequence of leaves, each exiting $W_s(n,\theta')$ twice by its upper side, which approach the bottom side. Case III is when the leaves are nested, and Case IV when they are not. There is thus an accumulation point of this sequence on $I_s$, that belongs to some leaf $F$. Given any compact portion $F'\subset F$, by \cite[Lemma 2.3]{BGG} there is a small tubular neighbourhood of $F'$ that is the image through a foliated chart of $(0,1)^2$ foliated horizontally. It follows that $F$ can neither enter $\wt{W_s}(n,\theta')$, nor stay on $I_r$, and it must exit $\wb{W_s(n,\theta)}$ by its upper side (since there are leaves doing the same arbitrarily close). Thus, $F$ goes right through $W_s(n,\theta')$, and $s\in B(n,\theta')$.            
   \end{itemize} 
   The lemma is proved.
\end{proof}

We can now prove Theorem \ref{thm31bis}.

\begin{proof}[Proof of Theorem \ref{thm31bis}]
Let $S$ be the set of $r\in A\cup B$ such that $K$ intersects $I_r$ (if $r\in A$) or $I_{r,1} \cup I_{r,2}$ (if $r\in B$). By compactness of $K$, $S$ is finite. So, thinking of $P_i$ ($i=1,2$) again as the upper half plane, there must be a rectangle (in both $P_i$) $(a,b)\times (0,c)$ that avoids $K$, such that $U\cap (a,b)\not=\emptyset$. Thus, up to replacing $A,B$ by $A\cap (a,b), B\cap (a,b)$, one can assume that $K$ is empty.  
  
Now, by Lemma \ref{31bis1}, considering $2P_{A\setminus D, B\setminus (E_1\cup E_2)}$ if necessary, we may assume that the foliation is nice.

Fix $\theta\in(0,\pi/2)$. By Lemma \ref{31bis3}, $B\subset \cup_{n\in\omega}B(n,\theta)$. Since $B$ is a $G_\delta$ that is dense in $U$, there must be some $n$ such that $B(n,\theta)$ is dense in some open $V\subset U$. Hence, by Lemma \ref{31bis4}, $\wb{V}\subset B(n,\theta')$ for all $\theta' >\theta$. Now if $r\in B(n,\theta')\cap(A\cup B)$, then $r$ satisfies $T(r,n,\theta')$, so the leaf $L_r$ must intersect $I_r$, which is clearly impossible if $r\in A$, since the foliation is nice. However, by the assumption on $A$, there is some $r\in A\cap \wb{V}$, a contradiction.
\end{proof}

\begin{rem} 
{\rm When $A$ is uncountable, $2P_{A,B}$ is not homotopy equivalent to a CW-complex by \cite[Theorem 1] {MesziguesBridges}. However, for any $A,B\subset\R$, the surface obtained by taking $P_{A,B}$ as above and attaching collars $(0,1)\times[0,1)$ to the $I_r$ for $r\in A$ is contractible (see \cite[Proposition 3]{MesziguesBridges}, already implicit in \cite{CalabiRosenlicht}). The proof of Theorem \ref{thm31bis} can easily be arranged to show that this collared version does not possess a foliation (for appropriate $A,B$). We thus have examples in two classes: a separable one which is not homotopy equivalent to a CW-complex, and a contractible non-separable one.}
\end{rem}

\medskip

The assumption that $A$ is (somewhere) dense is necessary: If $A\subset[0,1]$ is the classical Cantor ternary set and $B$ its complement in $[0,1]$, then $2P_{A,B}$ and $2P_{B,A}$ both admit foliations. The construction in the upper half plane is the following. For all $r\in A$, let $D_r$ be a disk of radius $1$ tangent to the bottom `boundary' at $(r,0)$. For all intervals $I$ in $B=\R \setminus A$, let $f_1^I < f_2^I < f_3^I$ be $\mathcal{C}^\infty$ curves tending to the extremities of $I$ with slope $0$ and avoiding all $D_r$ for $r\in A$. Such curves are depicted in Figure \ref{fig10} (where the superscripts $I$ are omitted), for an interval $I=(a,b)$. Notice that a curve tending to the bottom `boundary' of $P^o$ with slope 0 does not converge to any point in $P_{A,B}$. Foliate $P_{A,B}$ under each $f^I_3$ as in the same figure, and outside by horizontal lines. For $P_{B,A}$, we need only $f_1^I<f_3^I$, and we foliate below $f_3^I$ as in Figure \ref{fig11}, and put the vertical foliation outside. Taking the double in each case gives the expected foliated surface without boundary.

\begin{figure}[h]
\centering
    \epsfig{figure=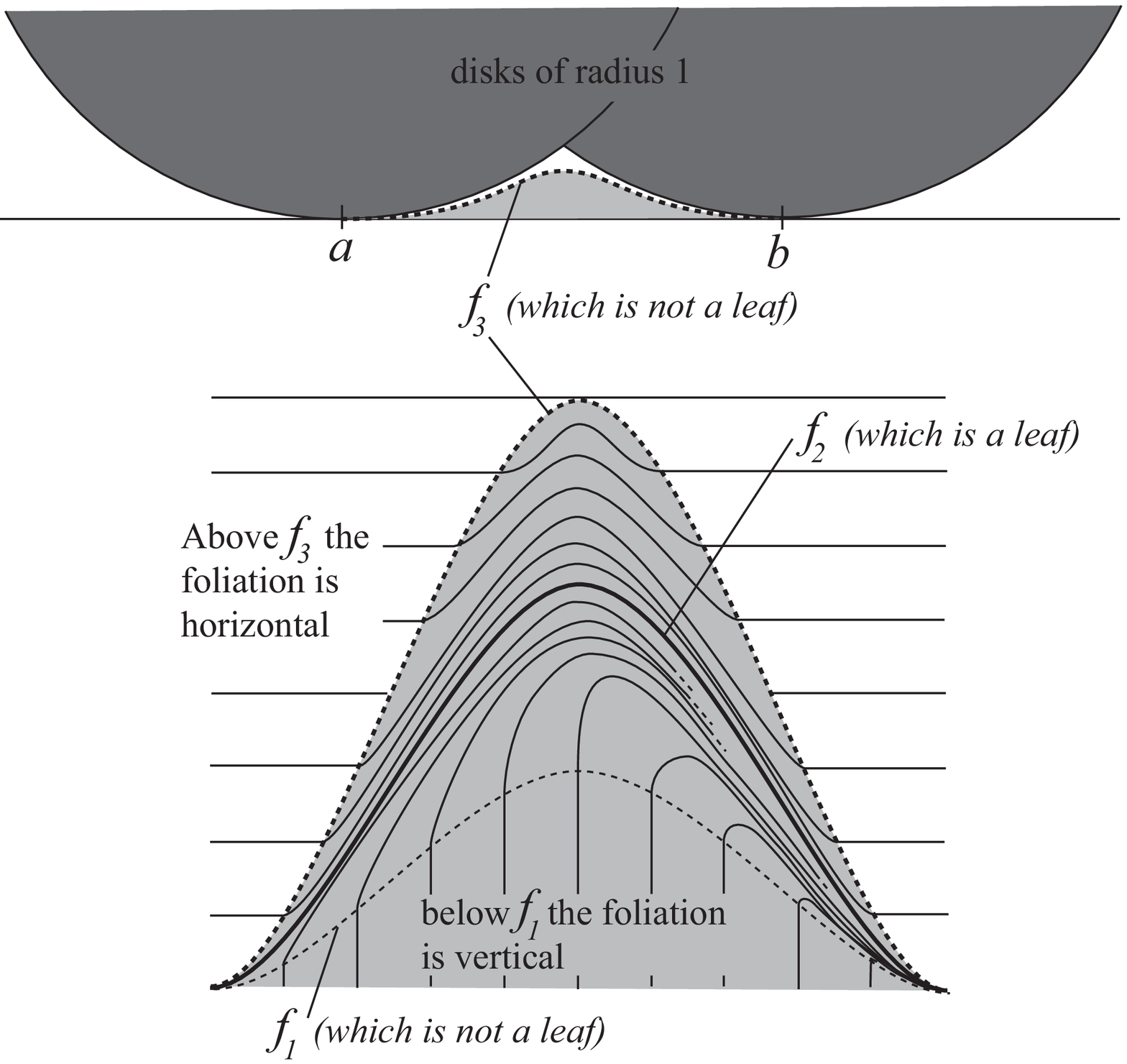, width=75mm}
    \caption{\label{fig10} Foliation on $P_{A,B}$, with $A$ the Cantor set, $B=[0,1] \setminus A$.}
\end{figure}

\begin{figure}[h]
\centering
    \epsfig{figure=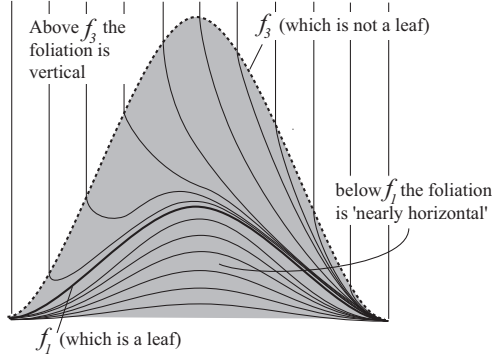, width=65mm}
    \caption{\label{fig11} Same as above, exchanging $A$ and $B$.}
\end{figure}

\subsection{A separable, simply connected surface lacking a foliation}\label{sep-simplconn-nofol}
In this subsection we construct a surface as in the title. Indeed, even removal of a compact subset does not allow a foliation.

Begin with the subset of the Nyikosization of the plane $\R^2$ described in Proposition \ref{thmkiekie} consisting of the copy of the long ray together with the quadrant $Q=\{(x,y)\in\R^2 : x>0 {\rm\ and\ } y<0\}$ but then remove all the points of the subset $\omega_1$ of the long ray. For each $\alpha\in\omega_1$ let $I_\alpha$ denote the open interval $(\alpha,\alpha+1)$ of the long ray. Now Moorize each of the intervals $I_\alpha$ around the point $\alpha + 1/2$ to get the desired manifold $M$. Denote by $I_\alpha'$ the half-open interval coming from the Moorization of $I_\alpha$.

Suppose that $M$ is foliated by $\mathcal F$. As in Lemma \ref{31bis1}, all but countably many of the intervals $I_\alpha'$ lie entirely in a leaf of $\mathcal F$. Each of these leaves then extends a positive distance into $Q$ and hence there is a positive integer $n$ so that uncountably many of these extensions reach the vertical line $x=1/n$; set $S=\{\alpha\in\omega_1 : I_\alpha' \mbox{ lies entirely in a leaf of } \mathcal F \mbox{ and its extension into } Q \mbox{ meets the line } x=1/n\}$. For each $\alpha\in S$ let $y_\alpha\in(-\infty,0)$ be such that $(1/n,y_\alpha)$ is the first point at which the leaf containing $I_\alpha'$ meets the line $x=1/n$. It is readily shown that if $\alpha<\beta$ then $y_\alpha<y_\beta$ as the leaves cannot cross. The sequence $\langle y_\alpha\rangle_{\alpha\in S}$ yields a contradiction because it is an uncountable, strictly increasing sequence in $\R$.

It is readily checked that $M$ is separable (any countable dense subset of $Q$ will do) and simply connected. A compact subset of $M$ will not interfere with the argument regarding non-foliability except to the extent that countably many more elements may need to be removed from $S$.




\section{Counting foliations on some $\omega$-bounded surfaces}

\subsection{A surface with a unique foliation}\label{sec-one-foliation}

Here we prove Proposition~\ref{one-foliation}, that is, the `paper cone' manifold $2\overline{Q}$ (where $\overline{Q}=\{(x,y)\in {\mathbb L}^2 : -y \le x\le y\}$) has only one foliation up to homeomorphism.

\begin{proof}[Proof of Proposition~\ref{one-foliation}] 
{\it Existence}: consider the obvious vertical foliation (illustrated on the left part of Figure~\ref{cornet}). {\it Uniqueness}: \cite[Proposition~6.1]{BGG} applied to each embedded copy of $Q$ in $2\overline{Q}$, reveals that the foliation is either asymptotically horizontal (H) or vertical (V). Let us analyse the three possible combinations (H-H), (H-V) and (V-V) of behaviours occurring on the `front-' respectively `back-side' of the doubled manifold.

\begin{figure}[h]
\centering
    \epsfig{figure=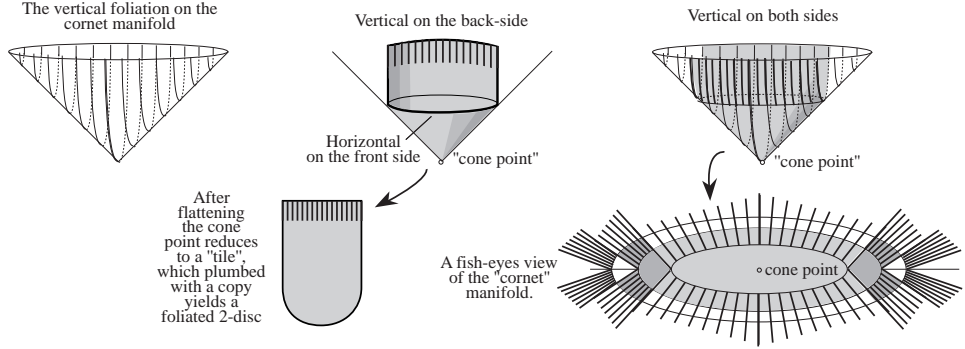,width=127mm }
    \caption{\label{cornet} The `paper cone' manifold: a surface with a unique foliation}
\end{figure}

$\bullet$ If (H-H), then after passing to the intersection of two closed unbounded sets, we obtain an impossible foliated $2$-disc.

$\bullet$ If (H-V), then we concentrate on the shaded region depicted on the central part of Figure~\ref{cornet}. Flattening this region yields the `tile' depicted right below, which in turn after plumbing with a copy along the `framed' boundary-data yields again a foliated $2$-disc.

$\bullet$ Finally when (V-V), then we identify easily compact subregions, in fact topological squares (see the different shadings on the right part of Figure~\ref{cornet}), each of which  admits a unique foliated extension according to \cite[Lemma~6.4]{BGG}. This completes the proof.
\end{proof}

\begin{rem} {\it (An origami avoiding any fish-eyes perception.)} {\rm A convenient trick to picture the manifold $2\overline{Q}$ involves cutting its back-side along the vertical axis then unfolding both octants via reflections at angles $\pi/4$ and $3\pi/4$. The cut manifold unfolds to cover the entire long upper half-plane ${\mathbb L} \times {\mathbb L}_{\ge 0}$, except that one has to identify $(x,0)$ with $(-x,0)$ for all $x\in{\mathbb L}_{\ge 0}$ (reminiscent of the Moorization process). This trick can be used to visualise without any fish-eye distortion the regions where we are applying \cite[Lemma~6.4]{BGG}. In fact, with this method it is also possible to analyse foliations on the bag-pipe surfaces $O_{g,n}$ of genus $g$ and ``pipus'' $n$, i.e. having $n$ pipes modelled on the conical pipes $2\overline{Q}\setminus\{ 0\}$.}

\end{rem}


\subsection{Simply-connected $\omega$-bounded surfaces with infinitely many foliations}\label{pipus-zero}

Here we prove Proposition~\ref{infinitely-many-foliations}.

Our example is built with the (closed) first octant of $\LL^2$ as the building block, defined as before by $ \OO=\{(x,y)\in\LL^2\,:\,0\le y\le x\}$. In our pictures, we draw an arrow from the horizontal to the diagonal in $\OO$, see Figure \ref{infa} (left).

\begin{figure}[h]
\centering
    \epsfig{figure=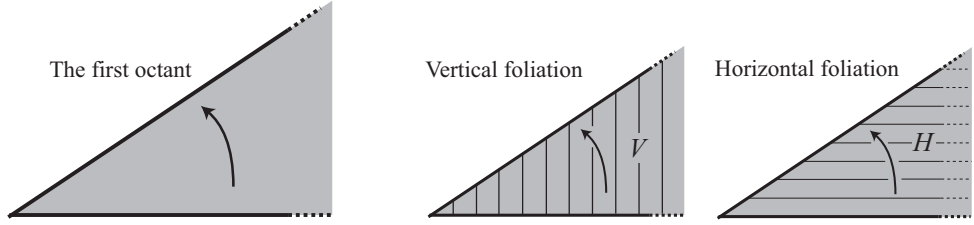,width=140mm}
    \caption{\label{infa} The octant and its trivial foliations.}
\end{figure}

The idea is the following (see Figure \ref{infb}, where one particular foliation is drawn as well on the surface). Take a countable number of copies of $\OO$, and then glue them together as follows. The $k$-th copy of $\OO$, which we label $\OO_k$, has a horizontal boundary $H_k$ and a diagonal boundary $D_k$. Consider the space given by the union of the $\OO_k$ where $H_{2k+1}$ is identified pointwise with $H_{2k+2}$ and $D_{2k}$ is identified with $D_{2k+1}$. (In fact, the union of $\OO_{2k}$ and $\OO_{2k+1}$ is a copy of $\LL_{\ge 0}^2$.) We then write $\Delta_k$ ($k=0,1,\dots$) for the images in the quotient of the old boundaries, counted counterclockwise (see Figure \ref{infb}, the thick lines). Say that $\Delta_k$ is a $2$-sided diagonal when $k$ is odd and a $2$-sided horizontal when $k$ is even. We then add a copy of $\LL_{\ge 0}$ that we call $\Delta_\omega$, to which the $\Delta_i$ accumulate, and another octant that we call $\OO_\omega$ with horizontal $\Delta_0$ and diagonal $\Delta_\omega$, to `close the circle' (again, see Figure \ref{infb}). Denote this surface by $S$. That $S$ is a longpipe follows from its definition, in fact one can easily write down explicitly $S=\cup_{\alpha\in\omega_1}S_\alpha$, with the $S_\alpha$ given by considering only the part of the octants where $x,y<\alpha$. Also, $S$ is simply connected.

The octant has two trivial foliations $\mathcal{H}$ and $\mathcal{V}$ given by the horizontal and the vertical lines (of course in the boundary there is a mixed tangent-transverse behavior), see Figure \ref{infa} (right). These foliations can be glued together to build infinitely many different foliations on $S$. The foliation $\mathcal F_i$ is drawn in Figure \ref{infb}. If one looks at the foliation restricted in the $\OO_k$, it will be horizontal in $\OO_0$, $\OO_{2i+1}$, $\OO_{2i+2}$ and $\OO_\omega$, and vertical in the other $\OO_k$.

\begin{figure}[h]
\centering
    \epsfig{figure=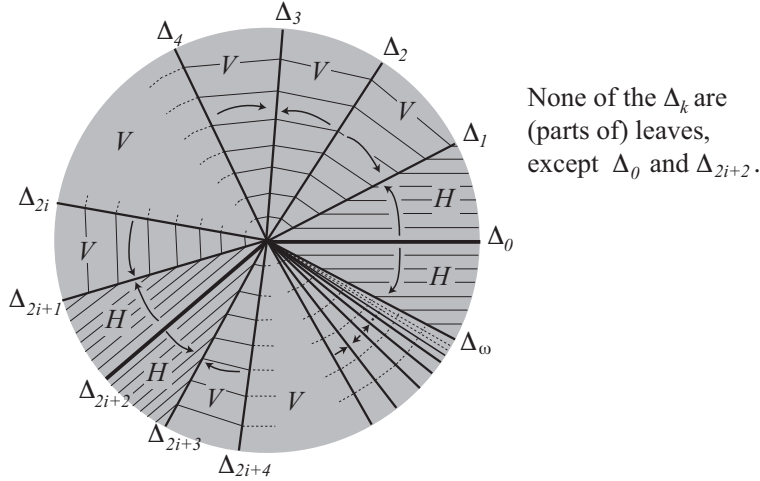,width=100mm}
    \caption{\label{infb} A surface with infinitely many foliations.}
\end{figure}

\begin{lemma}\label{lemmainfty}
There is no homeomorphism of $S$ that sends $\mathcal F_i$ to $\mathcal F_k$ if $i$ is different from $k$. 
\end{lemma}

\proof[Sketch of a proof]
We shall not give all the details since \cite{BaillifArxiv} contains many similar proofs. We say that some set in $S$ is unbounded if it is not contained in some $S_\alpha$.

The first step is to check that a homeomorphism $h$ must send $\Delta_\omega$ mainly to itself, that is, its image must intersect it on a closed unbounded set. This is because any copy of $\LL_+$ must intersect at least one of the $\OO_i$ ($i=0,1,\dots,\omega$) unboundedly. It is well-known that a copy of $\omega_1$ in the octant must either be contained in a horizontal (except possibly countably many points), or it intersects the diagonal on a closed unbounded set. Thus, $h$ will send a closed unbounded set of $\Delta_\omega$ into a horizontal or a diagonal of some $\OO_i$. But since these horizontals and diagonals do not have the same topological properties as $\Delta_\omega$ (to which others accumulate), the only possibility is as claimed.

By the same kind of arguments, one sees that the image of a $2$-sided diagonal by $h$ must intersect unboundedly another $2$-sided diagonal. It follows that the image of each $\Delta_i$ for $i$ odd intersects some $\Delta_{j(i)}$ (with $j(i)$ odd) unboundedly. Two closed unbounded sets of a diagonal always intersect, thus $j(i)\not= j(i')$ if $i\not= i'$. Moreover for each $j$ there is $i$ with $j=j(i)$. It follows that $j$ is increasing in $i$, and hence that $j(i)=i$.

Thus, a homeomorphism can basically just shake the $\Delta_i$'s (for odd $i$ and $i=\omega$), but not interchange them. The lemma follows immediately.
\endproof

\begin{rem}
{\rm The same kind of construction is possible but piling up $\omega_1$ copies of the octant instead of countably many. The resulting surface would then have at least $\omega_1$ foliations up to homeomorphism. Here, we have shown that $S$ has at least countably many foliations, but in fact there are probably no more, since an investigation similar to the one made for $\LL^2$ in \cite{BGG} should be possible, yielding that the given foliations are the only ones, up to homeomorphism. However we did not pursue this investigation.}
\end{rem}



\end{document}